\documentclass[12pt,a4paper]{amsart}
\usepackage[a4paper,inner=2.8cm,outer=2.5cm,top=2.8cm,bottom=2.5cm]{geometry}                
\usepackage{amsmath,amssymb,amscd,amsthm,amsfonts,anyfontsize,color,dsfont,enumerate,fix-cm,layout,lipsum,lpic,mathrsfs,mdwlist,pigpen,stmaryrd,tensor,tikz,thmtools,xspace}
\usepackage{hyperref}
\usepackage[all]{xy}
\usepackage{mathtools}
\mathtoolsset{showonlyrefs}
\usepackage{xfrac}
\theoremstyle{plain}                 
\newtheorem{theorem}{Theorem}[section]     
\newtheorem{proposition}[theorem]{Proposition} 
     
\newtheorem{lemma}[theorem]{Lemma}        
\newtheorem{conjecture}[theorem]{Conjecture}        
\theoremstyle{definition}           
\newtheorem{definition}[theorem]{Definition}    
 
\theoremstyle{remark}       
\newtheorem{remark}[theorem]{Remark}    
     
\hypersetup{colorlinks=true,linkcolor=blue,citecolor=purple,filecolor=magenta,urlcolor=cyan}

\newcommand{\mb}[1]{\mathbb{#1}} 

\newcommand{\mc}[1]{\mathcal{#1}}

\newcommand{\corc}[1]{\bigg{\langle} \, #1 \, \bigg{\rangle} ^{\circ}}
\newcommand{\cord}[1]{\bigg{\langle} \, #1 \, \bigg{\rangle} ^{\bullet}}
\DeclareRobustCommand{\Stirling}{\genfrac\{\}{0pt}{}}

\DeclarePairedDelimiter\floor{\lfloor}{\rfloor}

\DeclareMathOperator*\Res{Res}

\DeclareMathOperator{\Id}{Id}
\DeclareMathOperator{\End}{End}
\DeclareMathOperator{\Ad}{Ad}
\DeclareMathOperator{\ad}{ad}

\newcommand{\+}{\! + \!}
\newcommand{\mi}{\! - \!}

\newcommand{\sta}{\mathcal}

\newcommand{\MMMbar}{\overline{\sta M}}

\def\Z{\mathbb{Z}}
\def\N{\mathbb{N}}

\def\C{\mathbb{C}}

\def\oM{\overline{\mathcal{M}}}

\def\CP1{\mathbb{C}\mathrm{P}^1}

\allowdisplaybreaks[3]

\def\cA{\mathcal{A}}
\def\dmu{{[\mu]}}
\def\rmu{{\langle\mu\rangle}}
\long\def\B#1{\left ( #1 \right )}
\long\def\SB#1{\left [ #1 \right ]}

\def\commentOut#1{}
\def\quaramel{{Q_\mu^r(l)}}
\newcommand{\fixme}[1]{{\bf {\color{red}[#1]}}}
\def\?{\fixme{???}}
\def\nonNegInts{{\mathbb{Z}_{\geq 0}}}
\def\highPrePower{\dmu + s}
\def\highPower{r(\dmu + s)}
\DeclareMathOperator{\Poly}{Poly}
\def\PolyQTRmu{\Poly_{a,s,r}(\dmu)}

\begin{document}
\title[Towards an orbifold generalization of Zvonkine's $r$-ELSV formula]{Towards an orbifold generalization of Zvonkine's $r$-ELSV formula}

\author[R.~Kramer]{R.~Kramer}
\address{R.~K.: Korteweg-de Vries Institute for Mathematics, University of Amsterdam, Postbus 94248, 1090 GE Amsterdam, The Netherlands}
\email{R.Kramer@uva.nl}

\author[D.~Lewanski]{D.~Lewanski}
\address{D.~L.: Korteweg-de Vries Institute for Mathematics, University of Amsterdam, Postbus 94248, 1090 GE Amsterdam, The Netherlands}
\email{D.Lewanski@uva.nl}

\author[A.~Popolitov]{A.~Popolitov}
\address{A.~P.: Korteweg-de Vries Institute for Mathematics, University of Amsterdam, Postbus 94248, 1090 GE Amsterdam, The Netherlands}
\email{A.Popolitov@uva.nl}

\author[S.~Shadrin]{S.~Shadrin}
\address{S.~S.: Korteweg-de Vries Institute for Mathematics, University of Amsterdam,  Postbus 94248, 1090 GE Amsterdam, The Netherlands}
\email{S.Shadrin@uva.nl}

\begin{abstract} We perform a key step towards the proof of Zvonkine's conjectural $r$-ELSV formula that relates Hurwitz numbers with completed $(r+1)$-cycles to the geometry of the moduli spaces of the $r$-spin structures on curves: we prove the quasi-polynomiality property prescribed by Zvonkine's conjecture. Moreover, we propose an orbifold generalization of Zvonkine's conjecture and prove the quasi-polynomiality property in this case as well. In addition to that, we study the $(0,1)$- and $(0,2)$-functions in this generalized case and we show that these unstable cases are correctly reproduced by the spectral curve initial data. 
\end{abstract}

\maketitle

\tableofcontents

\section{Introduction}

\subsection{A recollection on the ELSV formula}
The celebrated Ekedahl-Lando-Shapiro-Vainshtein formula~\cite{ELSV01} connects single Hurwitz numbers to the intersection theory of the moduli spaces of curves. One of the possible interpretations of this formula comes from the theory of topological recursion~\cite{EyOr07}. Eynard proves in~\cite{Eyna11} that the ELSV formula is equivalent to the statement that the generating $n$-point functions of single Hurwitz numbers are expansions of the correlation forms obtained via topological recursion from the data of a particular spectral curve, which is in this case $\CP1$ equipped with
\begin{equation}
\label{eq:LambertCurve1}
x(z)=\log z -z,\quad y(z)=z,\quad B(z_1,z_2)=dz_1 dz_2 / (z_1-z_2)^2.
\end{equation}
The latter statement was known as the Bouchard-Mari\~no conjecture~\cite{BouchardMarino}, and  there are several earlier proofs of this conjecture that use the ELSV formula, but do not provide the equivalence statement, see~\cite{EynardMulaseSafnuk,MulaseZhang}. 

Let us discuss the precise mathematical content of the Bouchard-Mari\~no conjecture. Let $F_{g,n}:=\sum_{\mu_1,\dots,\mu_n} h_{g,\vec\mu} \exp(\sum_{i=1}^n \mu_i x_i)$ be the $n$-point generating function of the genus $g$ single Hurwitz numbers. Recall $x$, $y$, and $B$ defined in equation~\eqref{eq:LambertCurve1}. We have:
\begin{enumerate}
	\item The $1$-form $dF_{0,1}$ is the expansion of the $1$-form $ydx(z)$ in the variable $\exp(x)$ near the point $\exp(x)=0$.
	\item The symmetric $2$-form $[d_1\otimes d_2] F_{0,2}$ is the expansion of the $2$-form $B(z_1,z_2)$ in the variables $\exp(x_1),\exp(x_2)$ near the point $\exp(x_1)=\exp(x_2)=0$.
	\item The symmetric $n$-forms $[d_1\otimes\cdots\otimes d_n]F_{g,n}$, $2g-2+n>0$, can be represented as the expansions in $\exp(x_1),\dots,\exp(x_n)$ near the point $\exp(x_1)=\cdots=\exp(x_n)=0$ of the globally defined on $(\CP1)^{\times n}$ symmetric $n$-differentials $\omega_{g,n}(z_1,\dots,z_n)$. 
	\item These symmetric differentials $\omega_{g,n}$ satisfy the topological recursion for the initial data~\eqref{eq:LambertCurve1}.
\end{enumerate}
The first two statements can be checked by hand. The third statement follows from the shape of the ELSV formula. The ELSV formula implies that the coefficients of the $n$-point function $h_{g,\vec{\mu}}$ are equal to a certain explicit combinatorial factor which is not polynomial in the entries $\mu_i$, times a polynomial in the $\mu_i$. This property is often called quasi-polynomiality. The fourth statement is proved via analysis of the combinatorics of the cut-and-join equation for Hurwitz numbers in~\cite{EynardMulaseSafnuk,MulaseZhang}.
 Alternatively, in~\cite{Eyna11} the fourth statement is shown to be equivalent to the ELSV formula under assumption of the first three statements. 

The third statement above has a purely combinatorial interpretation and for a long time it was an open question whether it can be proved without reference to the ELSV formula. Two different proofs are now available, see~\cite{DKOSS13,KLS}. Once the third statement is proved independently, one can use the results in~\cite{EynardMulaseSafnuk,MulaseZhang} to prove the fourth statment, and then the equivalence of Eynard provides a new, purely combinatorial, proof of the ELSV formula~\cite{DKOSS13} (though the third statement requires some discussion of the analytic properties of the $n$-point functions).

\subsection{The \texorpdfstring{$q$}{q}-orbifold Hurwitz numbers}

The full story above can be repeated in the case of $q$-orbifold Hurwitz numbers, $q\geq 1$, which is a special case of double Hurwitz numbers, where one of the special fibers has monodromy consisting only of $q$-cycles. The orbifold analog of the ELSV formula is the Johnson-Pandharipande-Tseng formula~\cite{JohnsonPandharipandeTseng}. The topological recursion data in this case is $\CP1$ equipped with
\begin{equation}
\label{eq:LambertCurveQ}
x(z)=\log z -z^q,\quad y(z)=z^q,\quad B(z_1,z_2)=dz_1 dz_2 / (z_1-z_2)^2.
\end{equation}
As in the usual Hurwitz case, one can use the JPT formula to derive the quasi-po\-ly\-no\-mi\-a\-li\-ty of the $n$-point functions and then use the cut-and-join equation to prove the topological recursion~\cite{BHLM13,DLN12}. In particular, in these papers the $(0,1)$- and $(0,2)$-functions for $q$-orbifold Hurwitz numbers are related to the expansions of  $y \, dx(z)$ and $B(z_1,z_2)$ defined in equation~\eqref{eq:LambertCurveQ}. The equivalence of the topological recursion and the JPT formula is proved in~\cite{Lewanski2016}. An independent proof of the quasi-polynomiality property is given in~\cite{DLPS,KLS}, which gives a new, purely combinatorial, proof for the JPT formula~\cite{DLPS}.

\subsection{The \texorpdfstring{$r$}{r}-spin Hurwitz numbers}

There is another generalization of Hurwitz numbers~\cite{SSZ12} natural both from the point of view of the representation theory of the symmetric group~\cite{KerovOlshanski} and the Gromov-Witten theory of the projective line~\cite{OkPa06a}, where the typical singularity has the monodromy type of a completed $(r+1)$-cycle. 
These numbers are called the $r$-spin Hurwitz numbers, and this name is inspired by an ELSV-type formula, called the $r$-ELSV formula, conjectured by Zvonkine in 2006~\cite{Zvonkine2006}. This conjecture relates the $r$-spin Hurwitz numbers to the intersection numbers on the moduli spaces of $r$-spin structures, see~\cite{Zvonkine2006,SSZ}.

In this case the intersection number formula is only conjectural, and no alternative proof of the quasi-polynomiality is known. It is proved in~\cite{SSZ} that the conjectural $r$-ELSV formula is equivalent to the topological recursion on $\CP1$ for the following initial data:
\begin{equation}
\label{eq:LambertCurveR}
x(z)=\log z -z^r,\quad y(z)=z,\quad B(z_1,z_2)=dz_1 dz_2 / (z_1-z_2)^2.
\end{equation}
It is also proved in~\cite{MSS} that the differential of the $(0,1)$-function for $r$-spin Hurwitz numbers is indeed the expansion of $y \, dx(z)$ in the variable $\exp(x)$ near $\exp(x)=0$, where $x$ and $y$ are defined in equation~\eqref{eq:LambertCurveR}. 

The results of this paper include, as a special case, the proof that the $2$-differential obtained from the $(0,2)$-function of the $r$-spin Hurwitz numbers is given by the expansion of $B(z_1,z_2)$ in the variables $\exp(x_1),\exp(x_2)$ near the point $\exp(x_1)=\exp(x_2)=0$, where $x$ and $B$ are defined in equation~\eqref{eq:LambertCurveR}, as well as the quasi-polynomiality statement for the $(g,n)$-functions for $2g-2+n>0$.

\subsection{The \texorpdfstring{$q$}{q}-orbifold \texorpdfstring{$r$}{r}-spin Hurwitz numbers}

It is natural to combine the two generalizations of the concept of Hurwitz number: the monodromy of one special fiber consists of $q$-cycles, and the monodromy of the typical singularity is given by the completed $(r+1)$-cycle. This way we get $q$-orbifold $r$-spin Hurwitz numbers~\cite{MSS}. There is not much known about this generalization. There is only a quantum curve for this case that is proved in~\cite{MSS}. Note, however, that according to logic outlined in~\cite{AlLS15}, this leads to a guess of the spectral curve for this case, and the spectral curve implies an ELSV-type formula for this type of Hurwitz numbers as well.

The conjectural spectral curve in this case is $\CP1$ with the following initial data:
\begin{equation}
\label{eq:LambertCurveQR}
x(z)=\log z -z^{qr},\quad y(z)=z^q,\quad B(z_1,z_2)=dz_1 dz_2 / (z_1-z_2)^2.
\end{equation}
The result of~\cite{MSS} implies that the differential of the $(0,1)$-function is the expansion of $y \, dx(z)$ in $\exp(x)$ near the point $\exp(x)=0$, where $x$ and $y$ are defined in equation~\eqref{eq:LambertCurveQR}. 

The main result of this paper is the quasi-polynomiality statement for the $q$-orbifold $r$-spin Hurwitz numbers and the proof that the $2$-differential obtained from the $(0,2)$-function of the $q$-orbifold $r$-spin Hurwitz numbers is given by the expansion of $B(z_1,z_2)$ in the variables $\exp(x_1),\exp(x_2)$ near the point $\exp(x_1)=\exp(x_2)=0$, where $x$ and $B$ are defined in equation~\eqref{eq:LambertCurveQR}. We also prove the statement of~\cite{MSS} about the $(0,1)$-function in a new way. 

This allows us to generalize the conjecture of Zvonkine, in the following way. We conjecture that the $q$-orbifold $r$-spin Hurwitz numbers satisfy the topological recursion of the initial data given in equation~\eqref{eq:LambertCurveQR}. By the results of~\cite{Eyna11a,DOSS14} this immediately implies a conjectural ELSV-type formula for these Hurwitz numbers. The particular computation for the initial data~\eqref{eq:LambertCurveQR} is performed in~\cite{Lewanski2016}, where the correlation differentials for this spectral curve are presented in terms of the Chiodo classes~\cite{Chio08}. This allows us to obtain a very precise description of the conjectural ELSV-type formula for the $q$-orbifold $r$-spin Hurwitz numbers, which reduces in the case $q=1$ to the original conjecture of Zvonkine.

\subsection{Organization of the paper}
In section~\ref{sec:semi-infwedge} we give some necessary background on the semi-infinite wedge formalism. 
In section~\ref{sec:A-ops} we use the semi-infinite wedge formalism in order to define the $q$-orbifold $r$-spin Hurwitz numbers and to present them as the vacuum expectations of the so-called $\mathcal{A}$-operators. 
In section~\ref{sec:polynomiality} we prove the quasi-polynomiality property for the $q$-orbifold $r$-spin Hurwitz numbers.
In section~\ref{sec:unstable} we consider the unstable correlation differentials for the conjectural spectral curve and reproduce the $1$- and $2$-point functions for the $q$-orbifold $r$-spin Hurwitz numbers in genus $0$.
In section~\ref{sec:Zvonkinesconj} we describe precisely a conjectural ELSV-type formula for the $q$-orbifold $r$-spin Hurwitz numbers that generalizes the conjecture of Zvonkine for $r$-spin Hurwitz numbers.

\subsection{Acknowledgements}

The authors were supported by the Netherlands Organization for Scientific Research. S.~S. is grateful to D.~Zvonkine for the numerous helpful discussions of his 2006 conjecture.

\section{Semi-infinite wedge formalism}\label{sec:semi-infwedge}
This section introduces the notions of the semi-infinite wedge formalism, which we will use throughout this article. For a more complete introduction, see e.g. \cite{John15}. We will write \( \Z' \coloneqq \Z + \frac{1}{2}\) for the set of half-integers.\par
The main tool in this article will be the following algebra:
\begin{definition}
The Lie algebra \( \mathcal{A}_\infty \) is the \( \C \)-vector space of matrices \( (A_{ij})_{i,j \in \Z'} \) with only finitely many non-zero diagonals, together with the commutator bracket.\par
In this algebra, we will consider the following elements:
\begin{enumerate}
\item The standard (Schauder) basis of this algebra is the set \( \{ E_{i,j} \mid i,j \in \Z'\} \) such that \( (E_{i,j})_{k,l} = \delta_{i,k} \delta_{j,l} \);
\item The diagonal elements \( \mathcal{F}_n \coloneqq \sum_{k\in \Z'} k^n E_{k,k}\). In particular, \( C \coloneqq \mathcal{F}_0 \) is the \emph{charge operator} and \( E \coloneqq \mathcal{F}_1 \) is the \emph{energy operator}. An element \( A\) has energy \( e\in \Z \) if \( [A,E] = eA \);
\item For any non-zero integer \( n\), the energy \( n\) element \( \alpha_n \coloneqq \sum_{k \in \Z'} E_{k-n,k} \).
\end{enumerate}
\end{definition}
We will construct a certain projective representation of this algebra, called the infinite wedge space.
\begin{definition}
Let \( V\) be the vector space spanned by \( \Z'\): \( V = \bigoplus_{i \in \Z'} \C \underline{i} \), where the \( \underline{i} \) are a basis. We define the \emph{semi-infinite wedge space} \( \mathcal{V} \coloneqq \bigwedge^{\frac{\infty}{2}} V \) to be the span of all one-sided infinite wedge products
\begin{equation}\label{eq:semi-infwedge}
\underline{i_1} \wedge \underline{i_2} \wedge \dotsb, 
\end{equation}
with \( i_1 \in \Z'\), such that \( i_k + k -\frac12 = c\) for large \( k\), modulo the relations
\begin{equation}
\underline{i_1} \wedge \dotsb \wedge \underline{i_k} \wedge \underline{i_{k+1}} \wedge \dotsb = - \underline{i_1} \wedge \dotsb \wedge \underline{i_{k+1}} \wedge \underline{i_k} \wedge \dotsb.
\end{equation}
The constant \( c\) is called the \emph{charge}.
\end{definition}
A basis of \( \mathcal{V} \) is given by all elements of the form \eqref{eq:semi-infwedge} with the sequence \( (i_k) \) decreasing. 
\begin{remark}
Notice that \( \mathcal{A}_\infty \) has a natural representation on \( V\), but this cannot be extended to \( \mathcal{V} \) easily, as one would have to deal with infinite sums.
\end{remark}
\begin{definition}
For a partition \( \lambda \), define
\begin{equation}
v_\lambda \coloneqq \underline{\lambda_1 -\frac12} \wedge \underline{\lambda_2 - \frac32}  \wedge \dotsb.
\end{equation}
In particular, define the vacuum \( |0\rangle \coloneqq v_\emptyset \) and let the covacuum \( \langle 0| \) be its dual in \( \mathcal{V}^* \).\par
Define \( \mathcal{V}_0 \) to be the charge-zero subspace of \( \mathcal{V} \). Then \( \mathcal{V}_0 = \bigoplus_{n \in \N} \bigoplus_{\lambda \vdash n} \C v_\lambda \).
\end{definition}
\begin{definition}
For an endomorphism \( \mc{P} \) of \( \mathcal{V}_0 \), define its \emph{vacuum expectation value} or \emph{disconnected correlator} to be
\begin{equation}
\langle \mc{P} \rangle^\bullet \coloneqq \langle 0 | \mc{P} | 0 \rangle.
\end{equation}
\end{definition}
\begin{definition}
Define a projective representation of \( \mathcal{A}_\infty \) on \( \mathcal{V}_0\) as follows: for \( i \neq j\) or \( i = j > 0\), \( E_{i,j} \) checks whether \( v_\lambda \) contains \( \underline{j} \) as a factor and replaces it by \( \underline{i} \) if it does. If \( i = j < 0\), \( E_{i,i} v_\lambda = -v_\lambda \) if \( v_\lambda \) does not contain \( \underline{j} \). In all other cases it gives zero.\par
Equivalently, this gives a representation of the central extension \( \tilde{\mathcal{A}}_\infty = \mathcal{A}_\infty \oplus \C \Id \), with commutation between basis elements
\begin{equation}\label{eq:CommEE}
\left[E_{a,b}, E_{c,d}\right] = \delta_{b,c} E_{a,d} - \delta_{a,d} E_{c,b} +  \delta_{b,c}\delta_{a,d}(\delta_{b>0} - \delta_{d>0})\Id .
\end{equation}
\end{definition}
With these definitions, it is easy to see that \( C \) is identically zero on \( \mathcal{V}_0 \) and \( Ev_\lambda = |\lambda |v_\lambda \). Therefore, any positive-energy operator annihilates the vacuum. Similarly, so do all \( \mc{F}_r \).\par
Using this commutation rule, it is useful to compute:
\begin{lemma}\label{lem:CommgfEE}
\begin{align}
\left[\sum_{l \in \Z'} g_l E_{l-a,l}, \sum_{k \in \Z'} f_k E_{k-b,k}\right] &= \sum_{l \in \Z'} \left( g_{l-b} f_l - g_l f_{l-a} \right) E_{l - (a + b),l } \\
&+ \delta_{a+b}\delta_{a>0}(g_{1/2}f_{1/2-a} + \dots + g_{a-1/2} f_{-1/2}) \\
&+ \delta_{a+b}\delta_{b>0}(g_{1/2-b}f_{1/2} + \dots + g_{-1/2} f_{b-1/2}) .
\end{align}
\end{lemma}
In particular, with $g_l = f_k = 1$, it is possible to recover the usual commutation formula
\begin{equation}
[\alpha_a, \alpha_{b}] = a \delta_{a+b}.
\end{equation}

\section{\texorpdfstring{$\mathcal{A}$}{A}-operators}\label{sec:A-ops}

In this section, we will define the \( q\)-orbifold \( r\)-spin Hurwitz numbers as vacuum expectations of certain operators. We will then rewrite this expression to isolate the non-poynomial behaviour and get a formula for the supposed polynomial part as a vacuum expectation of a product of \( \mathcal{A} \)-operators. This line of thought originates from Okounkov and Pandharipande \cite{OP} and has also been used in e.g. \cite{Joh2,DLPS,KLS} to prove quasi-polynomiality of several different kinds of Hurwitz numbers.\par
We will write \( \mu = a[\mu ]_a + \langle \mu \rangle_a \) for the integral division of an integer \( \mu \) by a natural number \( a\). If \( a = qr\), we may omit the subscript.

\begin{definition}
The \emph{disconnected \( q\)-orbifold $r$-spin Hurwitz numbers} are
\begin{equation}\label{eq:defHurwcorr}
h_{g;\vec{\mu}}^{\bullet, q,r} \coloneqq  \cord{  \Big(\frac{\alpha_q}{q}\Big)^{\frac{|\mu |}{q}} \frac{1}{\big(\frac{|\mu |}{q}\big)!} \, \frac{ \mathcal{F}_{r+1}^b}{b!(r+1)^b} \, \prod_{i=1}^{l(\vec{\mu})}  \frac{\alpha_{-\mu_i}}{\mu_i} },
\end{equation}
where, by Riemann-Hurwitz, the number of $(r+1)$-completed cycles is 
\begin{equation}\label{eq:b-NumberMarkedPoints}
b = \frac{2g - 2 + l(\mu) + \frac{|\mu|}{q}}{r}.
\end{equation}
The \emph{connected \( q\)-orbifold \( r\)-spin Hurwitz numbers} \( h_{g;\vec{\mu}}^{\circ,q, r} \) are defined via the inclusion-exclusion formula from the disconnected ones.
\end{definition}
\begin{remark}
This formula can be interpreted as follows: we count covers of \( \mb{P}^1\), reading from \( 0 \) to \( \infty \). At the point \( 0\), we have ramification profile \( \mu \), corresponding to the product of \( \alpha \)'s on the right. The point \( \infty \) has \emph{orbifold} ramification, profile \( [q,q, \dotsc, q] \), corresponding to the \( \alpha\)'s on the left, divided by the extra symmetry factor \( \big( \frac{|\mu |}{q}\big)! \). In the middle, the ramification profiles are \emph{completed \( r+1\)-cycles}, corresponding to the \( \mc{F}_{r+1} \). These are formal linear combinations of ramification profiles, with `leading term' (most ramified) \( [r+1,1,\dotsc, 1]\), see \cite{OkPa06a}.
\end{remark}
\begin{definition}
The \emph{generating series of \( q\)-orbifold \( r\)-spin Hurwitz numbers} is defined as
\begin{equation}
H^{\bullet,q, r}(\vec{\mu}, u) \coloneqq \sum_{g=0}^{\infty} h_{g;\vec{\mu}}^{\bullet, q,r} u^{rb} = \cord{  e^{\frac{\alpha_q}{q}}\, e^{u^r\frac{\mathcal{F}_{r+1}}{r+1}}\, \prod_{i=1}^{l(\vec{\mu})} \frac{\alpha_{-\mu_i}}{\mu_i} }.
\end{equation}
The \emph{free energies} are defined as
\begin{equation}
F_{g,n}^{q,r}(x_1,\dotsc, x_n) \coloneqq \sum_{\mu_1, \dotsc, \mu_n =1}^\infty h_{g;\vec{\mu}}^{\circ, q,r} e^{\sum_{i=1}^n \mu_i x_i}
\end{equation}
\end{definition}

We now introduce \( \mathcal{A}\)-operators to capture the supposed quasi-polynomial behaviour of the \( q\)-orbifold \( r\)-spin Hurwitz numbers in the Fock space formalism.

\begin{definition}[\texorpdfstring{$\mathcal{A}$}{A}-operators]\label{Aoper}
\begin{align}
\mathcal{A}_{\langle \mu \rangle}^{q,r}(u, \mu) &\coloneqq \frac{1}{\mu}\!\! \sum_{\substack{l \in \Z + 1/2 \\ s \in \Z}} \frac{(u^r \mu)^s}{([\mu] +1)_{s}} \Bigg[ \sum_{t=0}^\infty \frac{\Delta_q^t}{q^t t!} \left( \frac{(l+\mu)^{r+1} - l^{r+1}}{\mu(r+1)} \right)^{s+[\mu]} \!\!\! E_{l+\mu-qt, l} \\
&\qquad \qquad \qquad \qquad + \delta_{\langle \mu \rangle_q,0} \sum_{j=1}^q \frac{\Delta_q^{[\mu ]_q - 1}}{q^{[\mu ]_q}[\mu ]_q!} \bigg( \frac{(l\+ \mu)^{r+1} - l^{r+1}}{\mu (r\+ 1)}\bigg)^{s+[\mu ]} \bigg{|}_{l = 1/2-j}\!\!\! \Id \Bigg],
\end{align}
where \( \Delta_q \) is the \( q\)-backward difference operator acting on functions of \( l\), i.e. \( (\Delta_q f)(l) = f(l) - f(l-q) \).
\end{definition}
\begin{remark}
In this definition, \( u \) is a formal variable, while \( \mu \)\textemdash{}at this point\textemdash{}is a positive integer. That is, for fixed \( \mu \),
\begin{equation}
\mc{A}_{\langle \mu \rangle}^{q,r}(u,\mu ) \in \mc{A}_{\infty} \llbracket u \rrbracket.
\end{equation}
Indeed, for fixed \( [\mu ] \) and fixed power of \( u\), \( t\) is bounded from above by \( r(s+ [\mu ] )\), so only finitely many diagonals are non-zero.
\end{remark}

These operators do indeed capture the conjectured polynomial behaviour, as is seen in the following proposition.

\begin{proposition}
\begin{equation}\label{eq:HurwA-op}
H^{\bullet,q, r}(\vec{\mu}, u) = \prod_{i=1}^{l(\vec{\mu})} \frac{(u^r\mu_i)^{[\mu_i]}}{[\mu_i]!}
\cord{ \prod_{i=1}^{l(\vec{\mu})} \mathcal{A}^{q,r}_{\langle \mu_i \rangle }(\mu_i,u) }.
\end{equation}
\end{proposition}

\begin{proof}
Since both $\mathcal{F}_{r+1}$ and $\alpha_q$ annihilate the vacuum, their exponents act as the identity operator on the vacuum. Hence we can write
\begin{equation}
H^{\bullet, q,r}(\vec{\mu}, u) =  \cord{ \prod_{i=1}^{l(\vec{\mu})} e^{\frac{\alpha_q}{q}}\, e^{\frac{u^r \mathcal{F}_{r+1}}{r+1}}\,  \frac{\alpha_{-\mu_i}}{\mu_i} \, e^{-\frac{u^r \mathcal{F}_{r+1}}{r+1}} \, e^{-\frac{\alpha_q}{q}} }.
\end{equation}

\begin{lemma}\label{lem:O-opcalc}
The conjugation with exponents of $\mathcal{F}$ reads
\begin{equation*}
\mathcal{O}_\mu (u) := e^{u^r\frac{ \mathcal{F}_{r+1}}{r+1}}\,  \alpha_{-\mu} \, e^{-u^r\frac{ \mathcal{F}_{r+1}}{r+1}} = \sum_{l \in \Z + 1/2} \sum_{s=0}^{\infty} \frac{(u^r \mu)^s}{s!} \left( \frac{(l+\mu)^{r+1} - l^{r+1}}{\mu(r+1)} \right)^s E_{l+\mu, l}.
\end{equation*}
\end{lemma}
\begin{proof}
As \( \Ad (e^X) = e^{\ad X} \), we have
\begin{equation*}
e^{u^r\frac{ \mathcal{F}_{r+1}}{r+1}}\,  \alpha_{-\mu} \, e^{-u^r\frac{ \mathcal{F}_{r+1}}{r+1}} = \sum_{s=0}^{\infty} \frac{u^{rs}}{(r+1)^ss!} \ad^s_{\mathcal{F}_{r+1}}\alpha_{-\mu}.
\end{equation*}
Applying lemma \ref{lem:CommgfEE} with $a=0$ and $g_l = l^{r+1}$, we see that every application of the operator $\ad_{\mathcal{F}_{r+1}}$ produces an extra factor $\left( (l+\mu)^{r+1} - l^{r+1} \right)$. Multiplying and dividing by $\mu^s$ yields the result.
\end{proof}

\begin{lemma}\label{lem:conjOalpha}
The conjugation with exponents of $\alpha_q$ is given as follows:
\begin{align}
\frac{1}{\mu}e^{\frac{\alpha_q}{q}} \mathcal{O}_\mu (u) e^{-\frac{\alpha_q}{q}} &= \sum_{l \in \Z + 1/2} \sum_{s=0}^{\infty} \frac{(u^r \mu)^s}{\mu \, s!} \sum_{t=0}\frac{\Delta_q^t}{q^t t!} \left( \frac{(l\+ \mu)^{r+1} - l^{r+1}}{\mu (r\+ 1)} \right)^s E_{l+\mu-qt, l} \\
&\quad+ \delta_{\langle \mu \rangle_q,0} \sum_{s=0}^\infty \frac{(u^r \mu)^s}{\mu \, s!} \sum_{j=1}^q \frac{\Delta_q^{[\mu ]_q - 1}}{q^{[\mu ]_q}[\mu ]_q!} \bigg( \frac{(l\+ \mu)^{r+1} - l^{r+1}}{\mu (r\+ 1)}\bigg)^s \bigg{|}_{l = 1/2-j} \!\!\! \Id.
\end{align}
\end{lemma}
\begin{proof}
Apply \( \Ad (e^X) = e^{\ad X} \) as before and lemma \ref{lem:CommgfEE} with $a=q$. The component of the identity can only occur if the total energy is zero, i.e. if \( \mu = qt \).
\end{proof}

Re-indexing $s \mapsto s + [\mu]$ we get the equation for the \( \mathcal{A} \)-operators, 
where we use that, for $s < -[\mu]$, the Pochhammer symbol vanishes, so we can extend the sum over all integers.
\end{proof}

\subsection{The inverse of the \texorpdfstring{$\mathcal{A}$}{A}-operators}

Following the schedule of \cite{KLS}, we would like to calculate the inverses of the \( \mathcal{A} \)-operators, viewed as elements of \( \End (V) \llbracket u\rrbracket \). This calculation starts with the observation that \( \alpha_\mu^{-1} = \alpha_{-\mu} \) as elements of \( \End (V)\). Conjugating this identity yields the following lemmata.

\begin{lemma}
\begin{equation}
\mathcal{O}_\mu (u)^{-1} = e^{u^r\frac{ \mathcal{F}_{r+1}}{r+1}} \alpha_{\mu} e^{-u^r\frac{ \mathcal{F}_{r+1}}{r+1}} = \sum_{l \in \Z + 1/2} \sum_{s=0}^{\infty} \frac{(u^r \mu)^s}{s!} \left( \frac{(l-\mu)^{r+1} - l^{r+1}}{\mu(r+1)} \right)^s E_{l-\mu, l}.
\end{equation}
\end{lemma}
\begin{proof}
This is completely analogous to the proof of lemma~\ref{lem:O-opcalc}, only changing the sign of \( \mu \) in appropriate places.
\end{proof}

\begin{lemma}
\begin{equation}
\mu e^{\frac{\alpha_q}{q}} \mathcal{O}_\mu (u)^{-1} e^{-\frac{\alpha_q}{q}} = \sum_{l \in \Z + 1/2} \sum_{s=0}^{\infty} \frac{\mu (u^r \mu)^s}{s!} \sum_{t=0}\frac{\Delta_q^t}{q^t t!} \left( \frac{(l-\mu)^{r+1} - l^{r+1}}{\mu(r+1)} \right)^s E_{l-\mu-qt, l}
\end{equation}
\end{lemma}
\begin{proof}
This is completely analogous to the proof of lemma~\ref{lem:conjOalpha}, bearing in mind that the coefficient of the identity is zero, as both operators in the repeated adjunction have positive energy.
\end{proof}

In defining the \( \mathcal{A}\)-operators, we extracted the coefficient 
\begin{equation}
\frac{(u^r\mu)^{[\mu ]}}{[\mu ]!}.
\end{equation}
Hence, the inverse of the \( \mathcal{A}\)-operators should include this factor. Therefore we get
\begin{lemma}
\begin{equation}\label{eq:A-inv}
\mathcal{A}_{\langle \mu \rangle}^{q,r}(u, \mu)^{-1} = \sum_{l \in \Z + 1/2} \sum_{s=0}^{\infty} \frac{\mu (u^r \mu)^{s+ [\mu ]}}{s![\mu ]!} \sum_{t=0}^\infty \frac{\Delta_q^t}{q^t t!} \left( \frac{(l-\mu)^{r+1} - l^{r+1}}{\mu(r+1)} \right)^s E_{l-\mu-qt, l}
\end{equation}
\end{lemma}

\section{Polynomiality}\label{sec:polynomiality}

\begin{definition} \label{def:poly-continuation}
An expression defined on a subset \( S \subset \C \) is \emph{polynomial} if there exists a polynomial $p$, defined on \( \C \), that agrees with this expression on $ S$. We then use $p$ as a definition of this expression at all other $x \in \mathbb{C}$.
\end{definition}

The goal of this section is to prove the following statement.

\begin{theorem}[Quasi-polynomiality]\label{thm:r-spinpoly}
For \( 2g-2+ \ell (\vec{\mu} ) > 0\), the \( q\)-orbifold \( r\)-spin Hurwitz numbers can be expressed in the following way:
\begin{equation}
h_{g,\vec{\mu}}^{\circ,q, r} = \prod_{i=1}^{l(\vec{\mu})} \frac{\mu_i^{[\mu_i]}}{[\mu_i]!}
P_{\langle \vec{\mu} \rangle}(\mu_1, \dots, \mu_{l(\vec{\mu})}),
\end{equation}
where $P$ are symmetric polynomials in the variables $\mu_1, \dots, \mu_{l(\vec{\mu})}$ whose coefficients depend on the parameters $\langle \mu_1 \rangle, \dots, \langle \mu_{l(\vec{\mu})} \rangle$.
\end{theorem} 

\begin{remark}
We prove that the degree of $P$ has a bound that does not depend on the entries of the partition $\vec{\mu}$. The actual computation of the degree in this case is difficult, and it is not necessary for the purpose of topological recursion. However, these numbers are expected to satisfy an ELSV-type formula (see conjecture~\ref{sec:Zvonkinesconj}). The conjecture would imply that the degree is equal to $3g-3+n$.
\end{remark}

\begin{remark}\label{rem:mu-floormu-polynomiality}
	Note that since we allow the coefficients of the polynomials  $P_{\langle \vec\mu \rangle}$ to depend on $\langle \vec\mu \rangle$, we can equivalently consider them as polynomials in $[\mu_1],\dots,[\mu_n]$, $n:=l(\vec\mu)$. The latter way is more convenient in the proof.
\end{remark}
Comparing the statement of theorem~\ref{thm:r-spinpoly} to equation~\eqref{eq:HurwA-op}, it is clear that the polynomials \( P\) must be the connected correlators of the \( \mathcal{A} \)-operators, defined via inclusion-exclusion from the disconnected versions. To prove this theorem, we will therefore first consider the disconnected correlators, and show that the coefficient of  a fixed power of \( u\) is a symmetric rational function in the \( \mu_i \), with only prescribed simple poles. The residues at these poles are explicitly related to the inverse \( \mathcal{A} \)-operators, and cancel in the inclusion-exclusion formula, proving quasi-polynomiality.\par
First we need some technical lemmata, analysing the dependence on \( \mu \) of single terms in the sums of the \( \mathcal{A} \)-operators.

\begin{lemma}\label{lem:RepDiffPoly}
The coefficients of the polynomial in \( l\), \( \frac{\Delta_q^{x+m}}{q^{x+m} (x+m)!} l^{p+x} \), are themselves polynomial in \( x\) for any \( p\) and \( m\). More precisely, the coefficient \( c_{m,a}^p(x) \) of \( l^a \) has degree \( 2p-a-2m\).
\end{lemma}
\begin{proof}
There is a version of the Leibniz rule for the backwards difference operator:
\begin{equation}
\Delta_q (fg)(l) = (\Delta_q f)(l)g(l) + f(l-q)(\Delta_q g)(l).
\end{equation}
Repeated application of this rule gives the following:
\begin{align}
\frac{\Delta_q^{x+m}}{q^{x+m}(x+m)!} l^{p+x} &= \sum_{i_0 + \dotsb + i_{x+m} = p-m} (l-q(x+m))^{i_{x+m}} \dotsb (l-q\cdot 0)^{i_0}\\
&= h_{p-m} (l-q(x+m), \dotsb, l) \\
&= \sum_{a=0}^{p-m} \binom{p +x}{a} h_{p-m-a}\big( -q,\dotsc, -q(x+m)\big) l^a\\
&= \sum_{a=0}^{p-m} \binom{p +x}{a}\Stirling{x+p-a}{x+m} (-q)^{p-m-a} l^a.
\end{align}
Here we used \cite[lemma 3.4]{KLS}. So the coefficient of \( l^a\) is given by
\begin{equation}
c_{m,a}^p (x) = (-q)^{p-m-a}\binom{p +x}{a}\Stirling{x+p-a}{x+m}.
\end{equation}
This binomial coefficient can be written as
\begin{equation}
\frac{1}{a!} (x+p) \dotsb (x+p-a+1),
\end{equation}
which is a polynomial in \( x\) of degree \( a\).\par
The Stirling number, on the other hand, requires a more subtle proof. Define \( f_t(x) = \Stirling{x+t}{x} \). We prove \( f_t \) is a  polynomial of degree \( 2t\) inductively on \( t\), starting with \( f_0(x) \equiv 1\).\par
For the induction step, recall the recursion relation for Stirling numbers, which can be written as follows:
\begin{equation}
\Stirling{x+t}{x} - \Stirling{x-1+t}{x-1} = x \Stirling{x-1+t}{x}.
\end{equation}
In other notation, \( (\Delta_1 f_t)(x) = x f_{t-1}(x) \). By induction, \( \Delta_1 f_t \) is polynomial of degree \( 2t-1\), hence \( f_t \) itself can be written as a polynomial of degree \( 2t\). The Stirling number we require is given by \( f_{p-a-m}(x+m)\), which is of degree \( 2(p-a-m)\). Adding degrees yields the result.
\end{proof}

\begin{remark}
Note that the equation \( \Delta_1 f = 0\) has non-polynomial solutions, e.g. \( f(x) = \sin (2 \pi x)\). However, we only prove that the functions in question can be represented as polynomials, not that there is no other analytic continuation.
\end{remark}



\begin{lemma}\label{lem:ratio}
  For fixed $r, i, s, \rmu \in \nonNegInts$ the expression
\begin{equation}
  \frac{\Delta_q^{i + [\mu ]_q}}{q^{i+[\mu ]_q}(i + [\mu ]_q)!} \bigg( \frac{(l+\mu )^{r+1} - l^{r+1}}{\mu (r+1)} \bigg)^{s+[\mu]}
\end{equation}
is polynomial in $[\mu ]$ (in the sense of definition \ref{def:poly-continuation}), of degree \( 2rs - 2i - 2 \langle [\mu ]_q\rangle_r \).
\end{lemma}

\begin{proof}
Expanding explicitly using Newton's binomial formula,
\begin{align}
\quaramel \coloneqq \frac{(l+\mu )^{r+1} - l^{r+1}}{\mu (r+1)} = \sum_{i=0}^r \binom{r+1}{i+1} \frac{\mu^i l^{r-i}}{(r+1)}.
\end{align}
Let us now consider the coefficient in front of $l^{r([\mu ]_{qr} + s) - a}$ for some particular values of ``offset'' $a$:
  \begin{align}
    \SB{l^{r([\mu ] + s) - 0}} \quaramel^{[\mu ] +s} & = 1; \\ \notag
    \SB{l^{r([\mu ] + s) - 1}} \quaramel^{[\mu ] +s} & = \binom{[\mu ] \+ s}{1} \commentOut{\B{1}^{\dmu-1}}
    \binom{r\+ 1}{2} \frac{\mu}{(r\+ 1)}; \\
    \SB{l^{r([\mu ] + s) - 2}} \quaramel^{[\mu ] +s} & = \binom{[\mu ] \+ s}{1} \commentOut{\B{1}^{\dmu-1}}
    \binom{r\+ 1}{3} \frac{\mu^2}{(r\+ 1)} + \binom{[\mu ] \+ s}{2} \commentOut{\B{1}^{\dmu-2}} {\binom{r\+ 1}{2}}^2 \frac{\mu^2}{(r\+ 1)^2}; \\
    \vdots \\
    \SB{l^{r([\mu ] + s) - a}} \quaramel^{[\mu ] +s} & =
    \sum_{\lambda \vdash a} \binom{[\mu ] \+ s}{\{ 
      \lambda^T_i \mi \lambda^T_{i+1} \}_{i\geq 1} }
    \Bigg( \prod_{i=1}^{\ell (\lambda)} \frac{1}{r+1}\binom{r\+ 1}{\lambda_i \+ 1} \Bigg)
   \mu^a,
  \end{align}
where the multinomial coefficient is
\begin{equation}
\binom{[\mu ] \+ s}{\{  \lambda^T_i \mi \lambda^T_{i+1} \}_{i\geq 1} } \coloneqq \frac{([\mu ] \+ s)!}{([\mu ]\+ s \mi \ell (\lambda ))! \prod_{i\geq 1} (\lambda_i^T \mi \lambda_{i+1}^T)! }.
\end{equation}
Clearly, this is a polynomial in $\dmu$ of degree $2a$\textemdash{}one $a$ comes from $\mu^a$ and the other from  the multinomial coefficient in the summand, corresponding to the partition $[1^a]$.

Furthermore, it has zeroes at $\dmu \in \mathbb{Z}_{\geq 0}$ for which   $\highPower - a < 0$ (i.e. when we want to extract a coefficient in front of the negative power of $l$). This is because the contributions of partitions $\lambda$ with more than $\dmu + s$ parts are zero thanks to the multinomial coefficient and partitions with $\ell (\lambda) \leq \dmu + s$ will have at least one part for which the corresponding binomial coefficient will be zero.

Let us denote
  \begin{align}
    \PolyQTRmu = \SB{l^{\highPower - a}} \quaramel^{\highPrePower}
  \end{align}
  Using lemma \ref{lem:RepDiffPoly}), denoting $i' = i + \langle [\mu ]_q \rangle_r $ for brevity and noting \( [\mu ]_q = r[\mu ] + \langle [ \mu ]_q \rangle_r \), we have
  \begin{align}
    \frac{\Delta_q^{i' + r\dmu}}{q^{i'+ r[\mu ]}(i' + r\dmu)!} Q_\mu^r(l)^{s+[\mu]}
    & = \frac{\Delta_q^{i' + r\dmu}}{q^{i'+r[\mu ]}(i' + r\dmu)!} \sum_{a=0}^{\highPower} \PolyQTRmu l^{\highPower - a} \\ \notag
    & = \sum_{a=0}^{r[\mu] + rs} \sum_{k=0}^{r s - i' - a} l^k c^{r s - a}_{i',k}(r\dmu) \PolyQTRmu \\ \notag
    & = \sum_{a=0}^{r s - i'} \sum_{k=0}^{r s - i' - a} l^k c^{r s - a}_{i',k}(r\dmu) \PolyQTRmu,
  \end{align}
where crucially in the last equality, we can choose upper summation limit of the first sum to be independent of $\dmu$.
We can do this, because:
\begin{itemize}
\item for $a > rs - i'$ the coefficients $c^{r s - a}_{i',k}(r\dmu)$ are zero;
\item for a particular value of $\dmu \in \mathbb{Z}_{\geq 0}$ it could happen that $r(\dmu + s) < r s - i'$. But we know that for $a > r(\dmu + s)$, $\PolyQTRmu = 0$. So, adding these zero terms does not change the sum.
\end{itemize}
We see that we have arrived at a manifestly polynomial expression, which completes the proof.\par
The degree follows as the degree of \( \PolyQTRmu \) is \( 2a \) and that of \( c_{i',k}^{rs-a} \) is \( 2(rs-a)-k-2i'\).
\end{proof}

These lemmata can be applied to prove the rationality of the disconnected correlators of \( \mathcal{A} \)-operators.

\begin{proposition} \label{prop:disconnected-rational}
For fixed power of $u$ and fixed \( [\mu_2],\dotsc, [\mu_n]\), and \( \langle \vec{\mu} \rangle \),
\begin{equation}
\cord{ \prod_{i=1}^{l(\vec{\mu})} \mathcal{A}_{\langle \mu_i \rangle }^{q,r}(\mu_i,u) }
\end{equation}
is a rational function in the variable $[\mu_1]$, with only simple poles at negative integers and at \( [\mu_1 ] = - \langle \mu \rangle \).
\end{proposition}
\begin{proof}
Let us make some observations about the following expression, where we write \( \mu = \mu_1 \),
\begin{equation}
\cord{\sum_{\substack{l \in \Z + 1/2 \\ s \in \Z}} \frac{(u^r \mu)^s}{\mu ([\mu] +1)_{s}} \sum_{t=0}\frac{\Delta_q^t}{q^t t!} Q_\mu^r(l)^{s+[\mu]} E_{l+\mu-qt, l} \prod_{j=2}^{l(\vec{\mu})} \mathcal{A}_{\langle \mu_i \rangle}^{q,r}(\mu_i, u) }
\end{equation}
First of all, the energy of the operators on the left should be positive, meaning that $\mu - qt < 0$. On the other side, the exponent of the finite difference operator cannot be greater than the degree of the polynomial to which it is applied, implying $t \leq r(s + [\mu])$. Combining these two restrictions, one obtains that $rs + r[\mu] \geq [\mu ]_q = r[\mu] + \langle [ \mu ]_q \rangle_r $. Solving for $s$ gives $s \geq \frac{\langle [\mu ]_q \rangle_r}{r} \geq 0$.\par
Moreover, the correlator is zero unless the sum of the energies is zero, which means 
\begin{equation} \label{eq:energy-is-zero}
(\mu - qt) + \sum_{j=2}^{l(\vec{\mu})} \mu_j - qt_j =0.
\end{equation}
Since the other $\mu_j$ are fixed, it is clear that $-i:=[\mu ]_q - t$ does not depend on $\mu$. We can rewrite the expression as
\begin{equation}\label{eq:firstArat}
\cord{\sum_{\substack{l \in \Z + 1/2 \\ s \geq 0}} \frac{(u^r \mu)^s}{\mu ([\mu] +1)_{s}} \sum_{i=0}^N \frac{\Delta_q^{i + [\mu ]_q}}{q^{i+[\mu ]_q}(i + [\mu ]_q )!} Q_\mu^r(l)^{s+[\mu]} E_{l + \langle \mu \rangle_q - qi, l} \prod_{j=2}^{l(\vec{\mu})} \mathcal{A}_{\langle \mu_i \rangle}(\mu_i, u) },
\end{equation}
where $N$ does not depend on $\mu$. Fixing the power of $u$ reduces the \( s\)-sum to a finite sum, as for the other \( \cA \)-operators the power of \( u\) is bouned from below by \( -[\mu_i]\). Now, the first fraction is clearly a rational function in $[\mu]$ while the second is polynomial by lemma~\ref{lem:ratio}. Hence, the entire correlator is a finite sum of rational functions, so it is rational itself.\par
The only possible poles can come from the Pochhammer symbol in the denominator, or the factor \( \frac{1}{\mu} \), and hence are at \( -s, 1-s, \dotsc, -1\) and at \( [\mu ] = -\frac{\langle \mu \rangle}{qr} \).
\end{proof}

To prove the connected correlator is a polynomial, we should therefore analyse these poles. As they are simple, we need only calculate the residues, which we do in the following proposition.

\begin{lemma}\label{lem:A-res}
The residue of the \( \mathcal{A} \)-operators at negative integers is, up to a linear multiplicative constant, equal to the inverse of the operator with a negative argument. More precisely,
\begin{align}
\Res_{\nu = - m} \mathcal{A}^{q,r}_\eta (u, \nu qr \+ \eta ) &= \frac{u^r}{m qr \mi \eta} \mathcal{A}^{q,r}_{-\eta}(u,mqr \mi \eta )^{-1} &&\text{if } 
\eta \neq 0 \label{eq:res1};\\
\Res_{\nu = - m} \mathcal{A}^{q,r}_0 (u, \nu qr) &= \frac{1}{mq^2r^2} \mathcal{A}^{q,r}_0 (u,mqr)^{-1} && \text{if } \eta = 0 \label{eq:res2}.
\end{align}
Here the residue is taken term-wise in the power series in \( u\), and the factor \( u^{-r} \) means a shift of terms.
\end{lemma}
\begin{remark}
Note that the first formula is slightly different from the one in \cite[Lemma 5.12]{KLS} in the case \( r=1\). This is because in that paper, an extra conjugation with \( u^{\frac{\mathcal{F}_1}{r}} \) was performed, resulting in different \( \cA\)-operators.
\end{remark}
\begin{proof}
Let us prove equations \eqref{eq:res1} and \eqref{eq:res2} together. The only contributing terms have \( s \geq m\), so we calculate
\begin{align}
\Res_{\nu = -m} \mathcal{A}^{q,r}_\eta (u, \mu )
&= \sum_{\substack{l \in \Z + 1/2 \\ s \geq m}} \frac{(u^r \mu)^s (\nu \+ m)}{\mu (\nu \+ 1)_{s}} \sum_{t=0}^\infty \frac{\Delta_q^t}{q^t t!} \left( \frac{(l+\mu)^{r+1} - l^{r+1}}{\mu(r+1)} \right)^{s+\nu} \!\! E_{l+\mu-qt, l} \bigg|_{\nu = -m}\\
&= \sum_{\substack{l \in \Z + 1/2 \\ s \geq m}} \frac{(u^r \mu)^s}{\mu (1 \mi m)_{m-1} (s \mi m)!} \sum_{t=0}^\infty \frac{\Delta_q^t}{q^t t!} \left( \frac{(l+\mu)^{r+1} - l^{r+1}}{\mu(r+1)} \right)^{s-m} \!\! E_{l+\mu-qt, l}\\
&= \sum_{\substack{l \in \Z + 1/2 \\ s \geq m}} \frac{(u^r \mu)^s (-1)^{m-1}}{\mu (m \mi 1)! (s \mi m)!} \sum_{t=0}^\infty \frac{\Delta_q^t}{q^t t!} \left( \frac{(l+\mu)^{r+1} - l^{r+1}}{\mu(r+1)} \right)^{s-m} \!\! E_{l+\mu-qt, l},
\end{align}
where we kept writing \( \mu \) for \( -mr + \eta \). As this is negative, however, it makes sense to rename it \( \mu = -\lambda \). Substituting and shifting the \( s\)-summation, we get
\begin{align}
\Res_{\nu = -m} \mathcal{A}^{q,r}_\eta (u, \mu )
&= \sum_{\substack{l \in \Z + 1/2 \\ s \geq m}} \frac{(-u^r \lambda )^s (-1)^{m-1}}{- \lambda (m \mi 1)! (s \mi m)!} \sum_{t=0}^\infty \frac{\Delta_q^t}{q^t t!} \left( \frac{(l - \lambda )^{r+1} - l^{r+1}}{-\lambda (r+1)} \right)^{s-m} \!\! E_{l-\lambda -qt, l}\\
&= \sum_{\substack{l \in \Z + 1/2 \\ s \geq m}} \frac{(u^r \lambda )^s}{\lambda (m \mi 1)! (s \mi m)!} \sum_{t=0}^\infty \frac{\Delta_q^t}{q^t t!} \left( \frac{(l - \lambda )^{r+1} - l^{r+1}}{\lambda (r+1)} \right)^{s-m} \!\! E_{l-\lambda -qt, l}\\
&= \sum_{\substack{l \in \Z + 1/2 \\ s \geq 0}} \frac{(u^r \lambda )^{s+m}}{\lambda (m \mi 1)! s!} \sum_{t=0}^\infty \frac{\Delta_q^t}{q^t t!} \left( \frac{(l - \lambda )^{r+1} - l^{r+1}}{\lambda (r+1)} \right)^s \!\! E_{l-\lambda -qt, l}.
\end{align}

Because \( \lambda = mr - \eta \), we have \( m = [\lambda ] + 1-\delta_{\eta 0} \) and \( \eta = - \langle \lambda \rangle \). Recalling equation \eqref{eq:A-inv}, we obtain the result.
\end{proof}
\begin{proof}[Proof of theorem~\ref{thm:r-spinpoly}]
The Hurwitz numbers are symmetric in their arguments, hence the \( P\) must be as well. By the same argument as for \cite[theorem 5.2]{KLS}, it suffices to prove polynomiality in the first argument.\par
Lemma \ref{lem:A-res} implies that we can express the residues in \( \mu_1 \) of the disconnected correlator as follows:
\begin{equation}
\Res_{\nu_1 = - m}\cord{ \prod_{i=1}^n \mathcal{A}_{\eta_i}(u, \mu_i)} = 
c(m,\eta_1 )
\cord{\mathcal{A}_{-\eta_1}(u,mqr-\eta_1 )^{-1} \prod_{i=2}^n \mathcal{A}_{\eta_i}(u,\mu_i)}.
\end{equation}
where \( c(m,\eta_1 )\) is the coefficient in lemma~\ref{lem:A-res}.
 Recalling equations~\eqref{eq:defHurwcorr} and \eqref{eq:HurwA-op} and realizing that the inverse \( \mathcal{A}\)-operator is given by the same conjugations as the normal \( \mathcal{A}\)-operator, but starting from \( \alpha_\mu \) instead of \( \alpha_{-\mu} \), we can see that this reduces to
\begin{equation}\label{eq:rescord}
\Res_{\nu_1 = - m}\cord{ \prod_{i=1}^n \mathcal{A}_{\eta_i}(u, \mu_i)} = C\cord{ e^{\frac{\alpha_q}{q}} \, e^{u^r \frac{\mathcal{F}_{r+1}}{r+1}} \alpha_{mqr-\eta_1} \prod_{i=2}^n \frac{\alpha_{-\mu_i}}{\mu_i} }
\end{equation}
for some specific coefficient \( C\) that depends only on $m$, $\eta_1$, and the \( \mu_i\).\par
Because \( [\alpha_k, \alpha_l] = k\delta_{k+l,0} \), and \( \alpha_{mqr-\eta_1} \) annihilates the vacuum, this residue is zero unless one of the \( \mu_i \) equals \( mqr - \eta_1 \) for \( i \geq 2\).\par
Now return to the connected correlator. It can be calculated from the disconnected one by the inclusion-exclusion principle, so in particular it is a finite sum of products of disconnected correlators. Hence the connected correlator is also a rational function in $\nu_1$, and all possible poles must be inherited from the disconnected correlators. So let us assume \( \mu_i = mqr-\eta_1 \) for some \( i \geq 2\). Then we get a contribution from \eqref{eq:rescord}, but this is canceled exactly by the term coming from
\begin{align}
\Res_{\nu_1 = - m}&\cord{\mathcal{A}_{\eta_1}(u, \mu_1)\mathcal{A}_{-\eta_1}(u,mqr-\eta_1) }\cord{\prod_{\substack{2 \leq j \leq n \\ j \neq i}} \mathcal{A}_{\eta_j}(u, \mu_j)} \\
&= C\cord{ e^{\frac{\alpha_{q}}{q}} e^{u^r\frac{\mathcal{F}_{r+1}}{r+1}} \alpha_{mqr-\eta_1} \alpha_{-(mqr-\eta_1)} }\cord{ e^{\frac{\alpha_{q}}{q}} e^{\frac{u^r \mathcal{F}_{r+1}}{r+1}} \prod_{\substack{2 \leq j \leq n \\ j \neq i}} \frac{\alpha_{-\mu_j}}{\mu_j} },
\end{align}
where the \emph{same} \( C\) occurs.\par
For the pole at \( [\mu ] = -\frac{\langle \mu \rangle}{qr} \), the only contributing term in equation~\eqref{eq:firstArat} has \( s=0\), so we get
\begin{equation}
\cord{\sum_{l \in \Z + 1/2 } \frac{1}{\mu} \sum_{i=0}^N \frac{\Delta_q^{i + [\mu ]_q}}{q^{i+[\mu ]_q}(i + [\mu ]_q )!} Q_\mu^r(l)^{[\mu]} E_{l + \langle \mu \rangle_q - qi, l} \prod_{j=2}^{l(\vec{\mu})} \mathcal{A}_{\langle \mu_i \rangle}(\mu_i, u) }.
\end{equation}
From the proof of lemma~\ref{lem:ratio}, we can clearly see that \( \mathop{\mathrm{Poly}}_{a,0,r}([\mu ]) \) is divisible by \( \mu \) if \( a > 0\), so we need \( a=0\) there. This implies we have only
\begin{equation}
c^0_{i',k} (r[\mu ])= (-q)^{-k-i'} \binom{r[\mu ]}{k} \Stirling{r[\mu ] - k}{r[\mu ]+i'},
\end{equation}
so we clearly need \( k = i'=0\), and thus \( i=0\) and \( \langle [ \mu ]_q \rangle_r= 0\). As the first \( \cA \)-operator acts on the covacuum, we still need \( qi - \langle \mu \rangle_q \geq 0\), so \( \langle \mu \rangle_q = 0\). As now \( \langle \mu \rangle_{qr} = \langle \mu \rangle_q + q\langle [\mu ]_q \rangle_r = 0\), we get that this term cancels against the same term from
\begin{equation}
\cord{\cA_0(u,\mu_1)}\cord{\prod_{i=2}^{\ell (\vec{\mu})} \cA_{\eta_i}(u,\mu_i)}.
\end{equation}
Hence, the connected correlator has no residues, which proves it is polynomial in \( \nu_1 \). Therefore, it is also a polynomial in \( \mu_1 \), see remark~\ref{rem:mu-floormu-polynomiality}. This completes the proof of the polynomiality in \( [\mu_1 ]\).\par
To be able to conclude that the connected correlator is polynomial in all $[\mu_1] \dots [\mu_n]$ we must show that the degree in $[\mu_1]$ of the connected correlator does not depend on $[\mu_2] \dots [\mu_n]$.\par
Since a connected correlator is a finite sum over products of disconnected 
correlators, given by the inclusion-exclusion formula, and the number of 
summands does not depend on $[\mu_2] \dots [\mu_n]$, the estimate on the degree of the connected correlator follows from estimates on degrees of disconnected correlators. The degree of the disconnected correlator, which is a rational function in $[\mu_1]$ by proposition~\ref{prop:disconnected-rational}, is defined as the leading exponent in the limit $[\mu_1] \rightarrow +\infty$.\par
Let us consider summands in the disconnected correlator \eqref{eq:firstArat} corresponding to a particular choice of \( s_j \geq -[\mu_j]\), for \( 2 \leq  j \leq n\). The contribution of genus $g$ covers is extracted by taking the coefficient in front of \( u^{2 g - 2 + n + \frac{1}{q} \sum_{i=1}^n \langle \mu_i \rangle}\), so we have
\begin{equation}
s = \frac{2 g - 2 + n}{r} + \frac{1}{rq} \sum_{i=1}^n \langle \mu_i \rangle - \sum_{j=2}^n s_j
\end{equation}
First of all, the factor $\frac{\mu^s}{\mu ([\mu] + 1)_s}$ contributes $-1$ to the degree. Then, by lemma~\ref{lem:ratio} the degree of
\begin{align}\label{eq:deltaQ}
\frac{\Delta_q^{i + [\mu ]_q}}{q^{i+[\mu ]_q}(i + [\mu ]_q)!} Q_\mu^r(l)^{s+[\mu]}
\end{align}
is $2 r s - 2 i - 2 \langle [\mu]_q \rangle$. It looks like the sum over $i$ in \eqref{eq:firstArat} goes from zero, so the highest degree of these polynomials depends on $[\mu_2] \dots [\mu_n]$ (through $s$ and estimates for $s_j$), but we are to obtain a finer estimate on the lower limit of summation.\par
We have
\begin{align}
t_j \leq r (s_j + [\mu_j]) \text{ for } 2 \leq j \leq n,
\end{align}
since exponents of difference operators cannot be greater than the exponent of the polynomials to which they are applied. Combined with the condition \eqref{eq:energy-is-zero} that the sum of the energies should be zero, this gives
\begin{align}
i \geq \frac{1}{q} \bigg( \langle \mu \rangle_q + \sum_{j=2}^n \langle \mu_j \rangle \bigg) - r \sum_{j=2}^n s_j,
\end{align}
which means that the degree of \eqref{eq:deltaQ} is bounded from above by
\begin{align}
2 (2 g - 2 + n) + \frac{2}{q} \sum_{i=1}^n \langle \mu_i \rangle - 2 \langle [\mu]_q \rangle_r  -\frac{2}{q} \bigg( \langle \mu \rangle_q + \sum_{j=2}^n \langle \mu_j \rangle \bigg)
= 2 (2 g - 2 + n),
\end{align}
which does not depend on $[\mu_2] \dots [\mu_n]$.\par
Thus, the degree of the disconnected correlator does not depend on $[\mu_2] \dots [\mu_n]$, and hence the degree of the connected correlator does not depend on $[\mu_2] \dots [\mu_n]$ either.
\end{proof}


\section{Computations for unstable correlation functions}\label{sec:unstable}
In this section we prove that the unstable correlation differentials for the spectral curve 
\begin{equation}\label{eq:spectralcurve}
\begin{cases}
X = e^x &= ze^{-z^{qr}}\\
y&=z^q
\end{cases}
\end{equation}
coincide with the expression derived from the $\mathcal{A}$-operators. The unstable $(0,1)$-energy was already derived in \cite{MSS} using the semi-infinite wedge formalism, we derive it here again to test our $\mathcal{A}$-operators. The computation for the unstable $(0,2)$-energy is a new result and fixes the ambiguity for the coordinate $z$ on the spectral curve.

\subsection{The case \texorpdfstring{$ (g,n)=(0,1)$}{(g,n)=(0,1)} } In this section we check that the spectral curve reproduces the correlation differential for $(g,n)=(0,1)$ obtained from the $\mathcal{A}$-operators. Explicitly, we show:
\begin{equation}\label{eq:dF}
dF^{q,r}_{0,1}(x) = y \, dx.
\end{equation}
Clearly, when dealing with a single $\mathcal{A}$-operator inside the correlator, only the coefficient of the identity operator contributes, since $\langle E_{i,j}\rangle = 0$. Hence, by definition~\ref{Aoper} and equation~\eqref{eq:HurwA-op}, we compute, using that connected and disconnected correlators are equal in this case:
\begin{align}
F^{q,r}_{0,1}(x) &\coloneqq \sum_{\mu = 1}^{\infty} [u^{-1 + \frac{\mu}{q}}]. H^{\circ,q,r}(u, \mu) e^{x\mu}\\
&= \sum_{\mu=1}^\infty \frac{\mu^{[\mu ]}}{[ \mu ]!} [u^{\frac{\mu}{q} -1}] \sum_{s=0}^\infty \frac{\delta_{\langle \mu \rangle_q,0}}{\mu} \frac{u^{r([\mu ]+s)}\mu^s}{([\mu ]+1)_s} \sum_{j=1}^q \frac{\Delta_q^{[\mu ]_q-1}}{q^{[\mu]_q} [\mu ]_q!} Q_\mu^r(l)^{[\mu ]+s} \bigg|_{l=\frac{1}{2}-j} e^{x\mu}\\
&= \sum_{m=1}^\infty \sum_{s=0}^\infty [u^{m-1}] \frac{u^{r([m]_r+s)}(mq)^{s+[m]_r-1}}{([m]_r+s)!} \sum_{j=1}^q \frac{\Delta_q^{m-1}}{q^m m!} Q_{mq}^r(l)^{[m]_r+s} \bigg|_{l=\frac{1}{2}-j} e^{xmq}\\
&= \sum_{n=0}^\infty \frac{\big( q(nr+1)\big)^{n-1}}{n!} \sum_{j=1}^q \frac{\Delta_q^{nr}}{q^{nr+1} (rn+1)!} Q_{(nr+1)q}^r(l)^n \bigg|_{l=\frac{1}{2}-j} e^{x(nr+1)q}\\
&= \sum_{n=0}^\infty \frac{\big( q(nr+1)\big)^{n-1}}{n!} \sum_{j=1}^q \frac{1}{q(rn+1)}\bigg|_{l=\frac{1}{2}-j} e^{x(nr+1)q}\\
&= q\sum_{n=0}^\infty \frac{\big( q(nr+1)\big)^{n-2}}{n!}e^{x(nr+1)q},
\end{align}
where the third line follows by setting \( \mu = mq\), the fourth line by setting \( m = nr+1\) and \( s=0\), and the fifth line because \( \frac{\Delta^d}{q^d d!} \) on a monic polynomial of degree \( d\) gives \( 1\).\par
As shown in \cite{MSS}, we have:
\begin{equation}
dF^{q,r}_{0,1}(x) = \left(\frac{W(-qre^{xqr})}{-qr}\right)^{1/r}dx,
\end{equation}
where $W$ is the Lambert curve $W(z) := -\sum_{n=1}^{\infty} \frac{n^{n-1}}{n!}(-z)^n$.
The properties of the Lambert curve (see \cite{MSS} for details) imply that the spectral curve \eqref{eq:spectralcurve} does satisfy equation~\eqref{eq:dF}, which can be shown by explicitly computing \( (ze^{-z^{qr}})^{qr} = e^{qrx}\).

\subsection{The case \texorpdfstring{$ (g,n)=(0,2)$}{(g,n)=(0,2)}}
In this section we prove that the $(0,2)$-correlation differential coincides with the usual Bergman kernel on the genus zero spectral curve.\par
Let us first compute the $(0,2)$-energy from the $\mathcal{A}$-operators.
\begin{lemma}\label{lem:F02}
\begin{equation}
F_{0,2}^{q,r}(X_1, X_2) =
\sum_{\substack{\mu_1, \mu_2=1 \\ qr|\mu_1 + \mu_2 \\ qr|\mu_1 }}^{\infty}   \frac{\mu_1^{[\mu_1]}}{[\mu_1]!}\frac{\mu_2^{[\mu_2]}}{[\mu_2]!} 
\frac{X_1^{\mu_1} X_2^{\mu_2}}{(\mu_1 + \mu_2)} + qr\sum_{\substack{\mu_1, \mu_2=1 \\ qr|\mu_1 + \mu_2 \\ qr\nmid \mu_1}}^{\infty}   \frac{\mu_1^{[\mu_1]}}{[\mu_1]!}\frac{\mu_2^{[\mu_2]}}{[\mu_2]!} 
\frac{X_1^{\mu_1} X_2^{\mu_2}}{(\mu_1 + \mu_2)}
\end{equation}
\begin{proof}
Let us write \( \mu \coloneqq \mu_1 + \mu_2 \).\par
By definition \ref{Aoper}, we have that
\begin{equation}\label{eq:F02A}
F_{0,2}^{q,r}(X_1, X_2) = \sum_{\mu_1, \mu_2 = 1}^{\infty} \frac{1}{\mu_1 \mu_2}\big[ u^{\frac{\mu}{q}}\big] \corc{\tilde{\mathcal{A}}(u, \mu_1)\tilde{\mathcal{A}}(u, \mu_2)} X^{\mu_1} X^{\mu_2},
\end{equation}
where
\begin{equation}
\tilde{\mathcal{A}}(u, \mu_i) = \sum_{l_i \in \Z + 1/2} \sum_{s_i=0}^{\infty} \frac{(u^r \mu_i)^{s_i}}{s_i!} \sum_{t_i=0}\frac{\Delta_q^{t_i}}{q^{t_i} t_i!} Q_{\mu_i}^r(l_i)^{s_i} \, E_{l_i+\mu_i-qt_i, l_i}.
\end{equation}
Note that the coefficient of the identity operator in $\tilde{\mathcal{A}}$ does not appear \textemdash\, indeed we are now interested in connected correlators and, in the case of $2$-points correlators, we have the simple relation $\langle\mathcal{A}_1\mathcal{A}_2\rangle^{\circ} = \langle\mathcal{A}_1\mathcal{A}_2\rangle^{\bullet} - \langle\mathcal{A}_1\rangle\langle\mathcal{A}_2\rangle.$
The contributions of the identity operators coincide precisely with the last summand.\par
Let us now make some observation about equation \eqref{eq:F02A}. Analysing the energy and the coefficient of \( u\), we find
\begin{equation}
\mu =q( t_1 + t_2) = qr(s_1 + s_2) \qquad \text{and} \qquad \mu_2 > qt_2 \geq 0.
\end{equation}
Moreover, the only term that can contribute in the correlator is the coefficient of the identity operator, produced by the commutation relation of $E$-operators described by formula \eqref{eq:CommEE}. Hence we compute that $F_{0,2}^{q,r}(X_1, X_2)$ is equal to 
\begin{equation}
\sum_{\mu_1, \mu_2 = 1}^\infty \sum_{s_1 + s_2 = \frac{\mu}{qr}} \sum_{\substack{t_1 + t_2 = \frac{\mu}{q} \\ 0 \leq qt_2 < \mu_2}} \! \! \! \! \! \sum_{l=\sfrac{1}{2}}^{\mu_2-qt_2-\sfrac{1}{2}} \frac{\mu_1^{s_1-1} \mu_2^{s_2-1}}{s_1! s_2!} \frac{\Delta_q^{t_1}}{q^{t_1} t_1!}Q_{\mu_1}^r(l)^{s_1}  \frac{\Delta_q^{t_2}}{q^{t_2}t_2!}Q_{\mu_2}^r(l - \mu_2 + t_2)^{s_2}    X_1^{\mu_1} X_2^{\mu_2}.
\end{equation}
Let us now observe that the sum of the degrees of the two difference operators equals the sum of the degrees of the polynomials to which they are applied. By lemma~\ref{lem:RepDiffPoly}, whenever the power of the difference operator is greater than the degree of the polynomial, the result equals zero. Hence the only nonvanishing terms should satisfy $t_1 = rs_1$ and $t_2 = rs_2$. We proved that $F_{0,2}^{q,r}(X_1, X_2)$ equals
$$ \sum_{\mu_1, \mu_2 = 1} \sum_{\substack{s_1, s_2 = 0 \\ s_1 + s_2 = \mu/qr}} \!\!\!\!\!(\mu_2 -qrs_2)\frac{\mu_1^{s_1-1}\mu_2^{s_2-1}}{s_1! s_2!}    X^{\mu_1} X^{\mu_2} \delta_{qrs_2<\mu_2}$$
We distinguish now two cases: the case in which the $\mu_i$ are divisible by $qr$ and the case in which the remainders are non-zero.

\subsubsection{Case $\mu_1 = qr\nu_1$}
In this case $\mu_2 = qr\nu_2$ and the Kronecker delta gives $s_2 = 0, \dots, \nu_2 -1$, which implies $s_1 = \nu_1 + 1, \dots, \nu_1 + \nu_2$. We split $(\mu_2 - qrs_2)$ in two terms, and remove the summand for $s_1 = \nu_1 + \nu_2$ from the sum. Writing $s$ for $s_1$, we get that the coefficient of \( X_1^{qr\nu_1}X_2^{qr\nu_2} \) is given by
\begin{equation}
(qr)^{\nu_1 + \nu_2-1} \Bigg[ \sum_{s=\nu_1 + 1}^{\nu_1 + \nu_2 - 1} \bigg( \frac{\nu_1^{s-1}\nu_2^{\nu_1 + \nu_2 - s}}{s! (\nu_1 + \nu_2 - s)!} - \frac{\nu_1^{s-1}\nu_2^{\nu_1 + \nu_2 - s-1}}{s! (\nu_1 + \nu_2 - s-1)!}\bigg) + \frac{\nu_1^{\nu_1 + \nu_2 -1}}{(\nu_1 + \nu_2)!} \Bigg].
\end{equation}
Multiplying and dividing by $(\nu_1 + \nu_2)!$ and collecting binomial coefficients we get
\begin{multline}
\frac{(qr)^{\nu_1 + \nu_2-1}}{(\nu_1 + \nu_2)!} \Bigg[ \sum_{s=\nu_1 + 1}^{\nu_1 + \nu_2-1} \bigg( \binom{\nu_1 \+ \nu_2}{s} \nu_1^{s-1}\nu_2^{\nu_1 \+ \nu_2 \mi s} - (\nu_1 \+ \nu_2)\binom{\nu_1 \+ \nu_2 \mi 1 }{s} \nu_1^{s-1}\nu_2^{\nu_1 + \nu_2 - (s+1)} \bigg) \\
+ \nu_1^{\nu_1 + \nu_2 -1}\Bigg].
\end{multline}
Distributing the factor \( (\nu_1 + \nu_2 \) and simplifying binomial coefficients, we get
\begin{equation}
\frac{(qr)^{\nu_1 + \nu_2-1}}{(\nu_1 + \nu_2)!} \Bigg[ \sum_{s=\nu_1 + 1}^{\nu_1 + \nu_2-1} \bigg( \binom{\nu_1 \+ \nu_2 \mi 1 }{s \mi 1} \nu_1^{s-1}\nu_2^{\nu_1 + \nu_2 - s} - \binom{\nu_1 \+ \nu_2 \mi 1 }{s} \nu_1^s\nu_2^{\nu_1 + \nu_2 - s-1} \bigg)+ \nu_1^{\nu_1 + \nu_2 -1}\Bigg].
\end{equation}
This is a telescoping sum, of which the only surviving term is
\begin{equation}
\frac{(qr)^{\nu_1 + \nu_2}}{qr(\nu_1 + \nu_2)} \frac{\nu_1^{\nu_1}}{\nu_1!}\frac{\nu_2^{\nu_2 - 1}}{(\nu_2 - 1)!} = \frac{1}{\mu_1 + \mu_2} \frac{\mu_1^{[\mu_1]}}{[\mu_1]!}\frac{\mu_2^{[\mu_2]}}{[\mu_2]!}.
\end{equation}

\subsubsection{Case $\mu_1 =qr\nu_1 + i$, with $0< i <qr$.}
In this case $\mu_2 = qr\nu_2 +(qr-i)$ and the Kronecker delta gives $s_2 = 0, \dots, \nu_2$, which implies $s_1 = \nu_1 + 1, \dots, \nu_1 + \nu_2 + 1$. We split $(\mu_2 - qrs_2)$ in two terms, and remove the summand for $s_1 = \nu_1 + \nu_2 + 1$ from the sum. Writing $s$ for $s_1$, the coefficient of \( X_1^{\mu_1}X_2^{\mu_2} \) equals
\begin{equation}
\sum_{s=\nu_1 + 1}^{\nu_1 + \nu_2 }  \Big[\frac{\mu_1^{s-1}}{s!}\frac{\mu_2^{\nu_1 + \nu_2 - s+1}}{(\nu_1 + \nu_2 - s+1)!} - qr\frac{\mu_1^{s-1}}{s!}\frac{\mu_2^{\nu_1 + \nu_2 - s}}{(\nu_1 + \nu_2 - s)!}\Big] + \frac{\mu_1^{\nu_1 + \nu_2}}{(\nu_1 + \nu_2 + 1)!}.
\end{equation}
The rest of the proof is completely analogous to the first case. The only remaining term is 
\begin{equation}
\frac{qr}{\mu_1 + \mu_2} \frac{\mu_1^{[\mu_1]}}{[\mu_1]!}\frac{\mu_2^{[\mu_2]}}{[\mu_2]!}.
\end{equation}
Summing up the first case and the second case for $i=1, \dots, qr-1$ yields the statement. This concludes the proof of the lemma.
\end{proof}
\end{lemma}

We are now armed to prove the main result of this section.
\begin{theorem}
\begin{equation}
\frac{dz_1dz_2}{(z_1 - z_2)^2} = \frac{dX_1dX_2}{(X_1 - X_2)^2} + d_1d_2F^{q,r}_{0,2}(X_1, X_2)
\end{equation}
\begin{proof}
It is enough to show that the Euler operator
\begin{equation}
E \coloneqq X_1 \frac{d}{dX_1} + X_2 \frac{d}{dX_2} = \frac{z_1}{1-qrz_1^{qr}}\frac{d}{dz_1} + \frac{z_2}{1-qrz_2^{qr}}\frac{d}{dz_2}
\end{equation}
applied to both sides of
\begin{equation}\label{eq:log_F}
\log(z_1 - z_2) = \log(X_1 - X_2) + F^{q,r}_{0,2}(X_1, X_2)
\end{equation}
gives equal expressions up to at most functions of a single variable $X_1C(X_1)$ and $X_2C(X_2)$. Let us compute the left hand side first:
\begin{align}
E&\log(z_1 - z_2)=\\
 &= \left( \frac{z_1}{1-qrz_1^{qr}} - \frac{z_2}{1-qrz_2^{qr}} \right)\frac{1}{z_1 - z_2}\\
&= 1 + \frac{1}{(1 - qrz_1^{qr})(1 - qrz_2^{qr})}\left( qr(z_1^{qr} + z_1^{qr-1}z_2 + \dots + z_1^{qr}) - (qr)^2z_1^{qr}z_2^{qr}\right)\\
&= 1 + \frac{d}{dx_1}\frac{d}{dx_2}\Bigg( qr \bigg(\frac{z_1^{qr} \log(z_2)}{qr} + \frac{z_1^{qr-1} z_2}{qr-1} + \frac{z_1^{qr-2} z_2^2}{2(qr-2)} + \dots + \frac{ \log(z_1)z_2^{qr}}{qr}\bigg) - z_1^{qr} z_2^{qr}\Bigg)\\
&= 1 +\frac{d^2}{dx_1 dx_2}\Bigg(z_1^{qr} x_2 +  x_1z_2^{qr} + qr \bigg(\frac{z_1^{qr-1} z_2}{qr-1} + \frac{z_1^{qr-2} z_2^2}{2(qr-2)} + \dots + \frac{z_1 z_2^{qr-1}}{qr-1}\bigg) + z_1^{qr} z_2^{qr}\Bigg)\\
 &= 1 + \sum_{\substack{k\geq 1\\ qr | k}} \frac{k^{[k]}}{[k]!} X_1^k + \! \sum_{\substack{l\geq 1\\ qr | l}} \frac{l^{[l]}}{[l]!} X_2^l + qr \!\!\!\!\! \sum_{\substack{\mu_1, \mu_2 \\ qr | \mu_1 + \mu_2 \\ qr \nmid \mu_1}}^{\infty} \frac{\mu_1^{[\mu_1]}}{[\mu_1]!}\frac{\mu_2^{[\mu_2]}}{[\mu_2]!}X_1^{\mu_1}X_2^{\mu_2} + \!\!\!\!
 \sum_{\substack{\mu_1, \mu_2 \\ qr | \mu_1 + \mu_2 \\ qr | \mu_1}}^{\infty} \frac{\mu_1^{[\mu_1]}}{[\mu_1]!}\frac{\mu_2^{[\mu_2]}}{[\mu_2]!}X_1^{\mu_1}X_2^{\mu_2}
\end{align}
where in the last equality we used the fact
\begin{align}
\frac{d}{dx}\left(\frac{z^i}{i}\right) &= \sum_{\mu: qr | \mu - i}^{\infty} \frac{\mu^{[\mu]}}{[\mu]!} X^{\mu} &\text{for  } i = 1, \dots, qr-1,\\
\frac{d}{dx} z^{qr} &= \sum_{\mu: qr | \mu}^\infty \frac{\mu^{[\mu ]}}{[\mu ]!}X^\mu, &
\end{align}
which was proved in \cite[Lemma 4.6]{SSZ}\textemdash{}substitute \( qr\) for \( r\) there.
By lemma \ref{lem:F02}, the right hand side reads:
\begin{align}
E \big(\log(X_1 - X_2) + F^{q,r}_{0,2}&(X_1, X_2)\big) =\\*
& 1 + qr \!\!\!\! \sum_{\substack{\mu_1, \mu_2 \\ qr | \mu_1 + \mu_2 \\ qr \nmid \mu_1 i}}^{\infty} \frac{\mu_1^{[\mu_1]}}{[\mu_1]!}\frac{\mu_2^{[\mu_2]}}{[\mu_2]!}X^{\mu_1}X^{\mu_2} + \!\!\!\! \sum_{\substack{\mu_1, \mu_2 \\ qr | \mu_1 + \mu_2 \\ qr | \mu_1}}^{\infty} \frac{\mu_1^{[\mu_1]}}{[\mu_1]!}\frac{\mu_2^{[\mu_2]}}{[\mu_2]!}X^{\mu_1}X^{\mu_2}
\end{align}
This concludes the proof of the theorem.
\end{proof}
\end{theorem}

\section{A generalization of Zvonkine's conjecture}\label{sec:Zvonkinesconj}

In this section we use the result of~\cite{Lewanski2016} in order to give a precise formulation of the orbifold version of Zvonkine's $r$-ELSV formula. 

\begin{conjecture} We propose the following formula for the $q$-orbifold $r$-spin Hurwitz numbers:
	\begin{align}
	h_{g,\mu_1,\dots,\mu_n}^{\circ,q, r} = &
\int_{\oM_{g,n}} \frac{\mathrm{C}_{g,n} \left (rq,q; qr-qr \left < \frac{\mu_1}{qr} \right >,
	\dots,
	qr - qr \left < \frac{\mu_n}{qr} \right > \right )
}{\prod_{j=1}^n (1-\frac{\mu_i}{qr}\psi_i)}
\\ \notag &
\times
r^{2g-2+n}(qr)^{\frac{(2g-2+n)q+\sum_{j=1}^n \mu_j}{qr}} \times \prod_{j=1}^{n} \frac{\left(\frac{\mu_j}{qr}\right)^{\floor*{\frac{\mu_j}{qr}}}}{\floor*{\frac{\mu_j}{qr}}!}.
\end{align}
\end{conjecture}
Here the class $\mathrm{C}_{g,n} \left (rq,q; qr-qr \left < \frac{\mu_1}{qr} \right >,
\dots,
qr - qr \left < \frac{\mu_n}{qr} \right > \right )$ is the Chiodo class~\cite{Chio08}. We use the same notation as in~\cite{Lewanski2016}, and we recall briefly its definition following the exposition there.

Let $\MMMbar_{g,n}^{qr,r}$ be the space of $qr$-th roots $S^{\otimes qr}\cong \omega_{\log}^{\otimes q} \left(\sum_{i=1}^n \left(qr \left < \frac{\mu_i}{qr}\right > -qr\right)x_i\right)$, $\omega_{\log}:=\omega(\sum_{i=1}^n x_i)$, on the curves $(C,x_1,\dots,x_n)\in\MMMbar_{g,n}$.
Note that the degree of the sheaf $\omega_{\log}^{\otimes q} \left(\sum_{i=1}^n \left(qr \left < \frac{\mu_i}{qr}\right > -qr\right)x_i\right)$ is equal to $q(2g-2+n)+qr\sum_{i=1}^n \left < \frac{\mu_n}{qr}\right > - nqr$ and is divisible by $qr$ (this follows from the Riemann-Hurwitz formula, that is, from the fact that $b$ given by equation~\eqref{eq:b-NumberMarkedPoints} is integer). 

We denote by $\pi\colon \mathcal{C}\to \MMMbar_{g,n}^{qr,r}$ the universal curve over $\MMMbar_{g,n}^{qr,r}$, with universal \( qr\)-th root line bundle \( \mc{S} \to \mc{C} \), and by $\epsilon\colon \MMMbar_{g,n}^{qr,r}\to \MMMbar_{g,n}$ the projection to the moduli space of curves. We define 
\begin{equation}
\mathrm{C}_{g,n} \left (rq,q; qr-qr \left < \frac{\mu_1}{qr} \right >,
\dots,
qr - qr \left < \frac{\mu_n}{qr} \right > \right ):=
\epsilon_*\left(c\left(R^1\pi_*\mc{S}\right)/c\left(R^0\pi_*\mc{S}\right)\right).
\end{equation}
This definition can be made very explicit, namely, there is an expression of the Chiodo classes in tautological classes via the Givental graphs. We refer to~\cite{2016arXiv160204705J,LewanskiSurvey} for a further discussion of the Chiodo classes.

In the special case $q=1$ this conjecture is reduced to Zvonkine's 2006 conjecture~\cite{Zvonkine2006}. In the case $r=1$ it is proved in~\cite{Lewanski2016} that this conjecture is equivalent to the Johnson-Pandharipande-Tseng formula first derived in~\cite{JohnsonPandharipandeTseng}. In the case $q=r=1$ this conjecture reduces to the ELSV formula first derived in~\cite{ELSV01}.

\bibliographystyle{alpha}
\bibliography{rspinlib}

\end{document}